\documentclass[]{interact}
\makeatletter
\newcommand{\monthyear}[1]{\def\@monthyear{\uppercase{#1}}}
\newcommand{\volnumber}[1]{\def\@volnumber{\uppercase{#1}}}
\makeatother
\newcommand{\N}{{\mathbb N}}  % Natural numbers
\newcommand{\Z}{{\mathbb Z}}  % Integers

\usepackage{array}
\usepackage{latexsym}
\usepackage{anyfontsize}
\usepackage{tikz}
\usepackage{url}
\usepackage{breakurl}
\usepackage{amsmath,amsthm,graphicx,color,enumerate,xcolor,spverbatim,setspace,amssymb}
\usepackage[all,2cell,ps]{xy}
\usepackage{amsfonts}
\usepackage{multicol}
\usepackage{soul}
\usepackage[
    colorlinks=False,     % colored links instead of boxes
    linkcolor=blue,      % color for \ref links
    citecolor=blue,       % color for \cite links
    urlcolor=purple,     % color for \url links
    bookmarks=true,       % create PDF bookmarks
    bookmarksopen=true,
    frenchlinks = False,
    linktoc = section,
    pagebackref = False
]{hyperref}
\usepackage[utf8]{inputenc}
\usepackage{subcaption}
\usetikzlibrary{shapes.geometric, arrows, positioning}

\newcommand{\bburl}[1]{\textcolor{blue}{\url{#1}}}

\tikzstyle{startstop} = [rectangle, rounded corners, minimum width=3cm, minimum height=1cm, text centered, draw=black, fill=white]
\tikzstyle{process} = [rectangle, minimum width=3cm, minimum height=1cm, text centered, draw=black, fill=white]
\tikzstyle{decision} = [diamond, minimum width=3cm, minimum height=1cm, text centered, draw=black, fill=white]
\tikzstyle{arrow} = [thick,->,>=stealth]

\let\emph\relax % there's no \RedeclareTextFontCommand
\DeclareTextFontCommand{\emph}{\em}     %I originally had \bfseries in here, but it looks bad

\newcommand{\vtr}[1]{\vec{\mathbf{#1}}}

\theoremstyle{plain}
\numberwithin{equation}{section} %change this to make globally numbered equations
\newtheorem{thm}{Theorem}[section] %remove [section] to make globally numbered environments
\newtheorem{theorem}[thm]{Theorem}
\newtheorem{lemma}[thm]{Lemma}
\newtheorem{example}[thm]{Example}
\newtheorem{definition}[thm]{Definition}
\newtheorem{proposition}[thm]{Proposition}

\newtheorem{remark}[thm]{Remark}
\newtheorem{question}[thm]{Question}
\newtheorem{conjecture}[thm]{Conjecture}

\numberwithin{table}{section} %change this and the following line to make globally numbered tables and figures
\numberwithin{figure}{section}

\newcommand{\X}{\vtr{X}}
\begin{document}

\received{Compiled October 8, 2025}

\title{General Recurrence Multidimensional Zeckendorf Representations}

\author{
\name{Jiarui Cheng\textsuperscript{a}\thanks{cheng.jiar@northeastern.edu},
Steven J. Miller\textsuperscript{b}\thanks{sjm1@williams.edu}
Sebastian Rodriguez-Labastida\textsuperscript{c}\thanks{sebasrodriguo@hotmail.com},
Tianyu Shen\textsuperscript{d}\thanks{1790299332@shu.edu.cn},
Alan Sun\textsuperscript{e}\thanks{alansun@umich.edu},
and Garrett Tresch\textsuperscript{f}\thanks{treschgd@tamu.edu}}
\affil{
        \textsuperscript{a}Northeastern University, Boston, MA 02115;\\
        \textsuperscript{b}Department of Mathematics, Williams College, Williamstown, MA 01267;\\
       \textsuperscript{c}Universidad Panamericana, CDMX 03920;\\
       \textsuperscript{d}Shanghai University, Shanghai, China, 200444;\\
       \textsuperscript{e}University of Michigan, Ann Arbor, MI 48109;\\
       \textsuperscript{f}Department of Mathematics, Texas A\&M University, College Station, TX 77840}
}

\maketitle

{\bf Article type}: research

\bigskip

\begin{abstract}
    We present a multidimensional generalization of Zeckendorf's Theorem (any positive integer can be written uniquely as a sum of non-adjacent Fibonacci numbers) to a large family of linear recurrences. This extends work of Anderson and Bicknell-Johnson in the multi-dimensional case when the underlying recurrence is the same as the Fibonacci one. Our extension applies to linear recurrence relations defined by vectors $\vec{\mathbf{c}} = (c_1, c_2, \ldots, c_k)$ such that $c_1\geq c_2\geq\cdots \geq c_k$ and where $c_k = 1$. Under these conditions, we prove that every integer vector in $\mathbb{Z}^{k-1}$ admits a unique $\vec{\mathbf{c}}$-satisfying representation ($\vec{\mathbf{c}}$-SR) as a linear combination of vectors, $(\vec{\mathbf{X}}_n)_{n\in \mathbb{Z}}$ defined for every $n\in \mathbb{Z}$ initially by zero and standard unit vectors and then the recursion $$\vec{\mathbf{X}}_{n} := c_1\vec{\mathbf{X}}_{n -1} + c_2\vec{\mathbf{X}}_{n - 2} + \cdots + c_k\vec{\mathbf{X}}_{n-k}.$$
  To establish this, we introduce carrying and borrowing operations that use the defining recursion to transform any $\vec{\mathbf{c}}$ representation into a $\vec{\mathbf{c}}$-SR while preserving the underlying vector. Then, by establishing bijections with properties of scalar Positive Linear Recurrence Sequences (PLRS), we prove that these multidimensional decompositions inherit various properties, such as the number of summands exhibits Gaussian behavior and summand minimality of $\vec{\mathbf{c}}$-SRs over all $\vec{\mathbf{c}}$-representations.\\

  \noindent \textbf{DEDICATION}\\
  It is a pleasure to dedicate this paper to Peter Anderson and Marjorie Bicknell-Johnson, both for their work which inspired this project as part of the 2025 Polymath Jr program, and for their service to the Fibonacci Association. \\
\end{abstract}


%\newpage
\tableofcontents

\section{Introduction and Preliminaries}
\subsection{Introduction}
\indent \indent The Fibonacci numbers have inspired results of enduring interest for over 2000 years. An interesting recent one is Zeckendorf's Theorem \textcolor{blue}{\cite{Zeckendorf72}}, which states that every positive integer can be represented uniquely as a sum of non-consecutive Fibonacci\footnote{We define the Fibonacci numbers by $F_1 = 1,  F_2 = 2$, and $F_{n}=F_{n-1}+F_{n-2}$ for $n\geq 3$.} numbers $(F_n)_{n=1}^\infty$. Since then, Zeckendorf's theorem has been extended to a large family of recurrence sequences by first specifying a rule to guarantee a unique representation and then deducing the structure of the sequence (see \textcolor{blue}{\cite{ABJ11, BCCSW, BDEMMTTW, BILMT2, CFHMN1, DDKMMV14, DFFHMPP, GTNP, Ha, KKMW11, MW12, MW2}} and the theorems therein). We take particular interest in linear recurrence sequences (LRS's). Given a vector $(c_1, \ \ldots , \ c_k) \in \Z^{k}$ and some proper initial terms, a LRS, $(X_n)_{n\in \N}$, is defined by the relation
$$X_{n+1} \ =\ c_1 X_{n} + \cdots + c_k X_{n+1-k}$$
for all $n\geq k$. In this case, we call the underlying vector $\vtr{c}=(c_1, \ \ldots, \ c_k)$ the \textit{recurrence vector}. Fraenkel \textcolor{blue}{\cite{Fr}} generalized Zeckendorf's result to all linear recurrences with weakly decreasing recurrence vectors (see Definition \textcolor{blue}{\ref{Def: weakly dec}}). More recently, Miller and Wang \textcolor{blue}{\cite{MW12, MW2}}, and independently Hamlin \textcolor{blue}{\cite{Ha}}, have generalized Zeckendorf's theorem to LRS's with nonnegative recurrence vectors $\vtr{c}$ with the additional restriction that $c_1 \geq 1$. Such recurrence sequences are called \emph{Positive Linear Recurrence Sequences} or simply \emph{PLRS's} (see Definition \textcolor{blue}{\ref{def:c-rec seq}}). Furthermore, various authors \textcolor{blue}{\cite{Ha, CFHMN1, CFHMNPX}} have provided evidence that this is the broadest class of LRS's for which one can expect Zeckendorf's theorem to extend in a simple manner. We focus on PLRS's, for which the representation rule is the notion of a \textit{legal decomposition} that formalizes when a representation of a positive integer over the PLRS is not able to be reduced using the underlying recurrence relationship. The restriction to legal representations then allows for unique representations of positive integers over linear combinations of the underlying PLRS.\\

In \textcolor{blue}{\cite{ABJ11}} Anderson and Bicknell-Johnson transfer Zeckendorf's Theorem into the multidimensional setting. In particular, they define a sequence of vectors $(\vtr{X}_n)_{n\in \Z}$ in $\Z^{k-1}$ and show that vectors of $\Z^{k-1}$ can each be uniquely represented as the sum of elements from $(\vtr{X}_{-n})_{n=1}^\infty$ where this representation does not contain $k$ consecutive elements of $(\vtr{X}_{-n})_{n=1}^\infty$. We extend these results by transferring PLRS's with weakly decreasing recurrence vectors to the multidimensional setting. However, this multidimensional extension is not a straightforward application of the techniques of \textcolor{blue}{\cite{ABJ11}} to general PLRS's. In fact, Remark \textcolor{blue}{\ref{rem: c_k=1}} illustrates that not all of these vectorized PLRS's will produce integer vectors while Example \textcolor{blue}{\ref{Ex: Bad Algorthm}} demonstrates the standard algorithm for finding representations of the desired form by carrying and borrowing terms across the recursion fails to terminate when $\vtr{c}=(1,3,1)$. The restriction to PLRS's with weakly decreasing recurrence vectors\footnote{These restrictions to weakly decreasing recurrence vectors can be found in many other results related to recurrence sequences. In addition to \textcolor{blue}{\cite{Fr}}, in \textcolor{blue}{\cite{CHHMPV18b}} the same constraints are used to guarantee summand minimality (see Section \textcolor{blue}{\ref{sec: SM}}).} such that $c_k=1$ remedies the issues seen in both Remark \textcolor{blue}{\ref{rem: c_k=1}} and Example \textcolor{blue}{\ref{Ex: Bad Algorthm}} as the constraint guarantees the termination of a process that converts an arbitrary multidimensional decomposition over a vectorized PLRS into one that generalizes legal decompositions and shows uniqueness.

\subsection{Preliminary Definitions}

Throughout this paper, $k \geq 2$ is a fixed integer and $\vtr{c} = (c_1,\ldots, c_k)$ is an integer vector such that $c_1>0$, $c_2,\ldots, c_{k-1} \geq 0$ and $c_k=1$. In addition, we use the notation for sequences, $(a_n)_{n=1}^\infty$, and infinite strings, $a_1a_2\ldots$, interchangeably.
% Credit: Alan

The following definition is identical to that of a Positive Linear Recurrence Sequence (PLRS) coined and studied in \textcolor{blue}{\cite{KKMW11}}, though they do not restrict to $c_k=1$.

\begin{definition}{\rm{(}\emph{PLRS}, \textcolor{blue}{\cite{KKMW11}}\rm{)}}\label{def:c-rec seq}
    We say $(X_n)_{n=-k+2}^{\infty} \subseteq \mathbb{Z}$ is a \textbf{$\vtr{c}$-recursive sequence} (or \textbf{{$\vtr{c}$-recurrence}}) if the following conditions hold:
    \begin{enumerate}

        \item[(1)] $X_{n}=0$ for $-k+2\leq n\leq 0$,\\
        
        \item[(2)] $X_1 = 1$ and for all $n = 2, 3,\ \ldots,\ k$,
        \[X_n\ =\ c_1X_{n - 1} + c_2X_{n - 2} + \cdots + c_{n - 1}X_1 + 1 \text{, \ and}\]

        \item[(3)] when $n > k$, the following recurrence is satisfied: \[X_n\ =\ c_1X_{n - 1} + c_2X_{n - 2} + \cdots + c_kX_{n - k}.\]
    \end{enumerate}
\end{definition}

\noindent

\begin{remark}\label{Rem: Fib and Custom rec.}
      Note that in the case that $c_n = 1$ for all $n$, $(X_n)_{n=1}^\infty$ exhibits the $k$-bonacci recurrence $($see \rm{\textcolor{blue}{\cite{ABJ11}}}\ \emph{where the $k$-bonacci sequence is defined with slightly different initial terms$)$. In particular, if in addition $k=2$ then $(X_n)_{n=1}^\infty$ is the Fibonacci Sequence.}
\end{remark}

\begin{example}\label{Ex: PLRS seq}
    \begin{itemize}
        \item[]
       \item The Tribonacci Recurrence $[ k=3$, $\vec{c} = (1, 1, 1)].$            \begin{itemize}
            \item[]
                \item Initial nonzero terms:
                    \begin{align*}
                        X_1 &\ =\ 1, \\
                        X_2 &\ =\ 1 \cdot 1 + 1\ =\ 2, \\
                        X_3 &\ =\  1 \cdot 2 + 1 \cdot 1 + 1\ =\ 4.
                    \end{align*}

                \item Recurrence for $n > 3$:
                    \[ X_n\ =\ X_{n-1} + X_{n-2} + X_{n-3}. \]
                \item Sequence: $1, 2, 4, 7, 13, 24, 44, \ldots$\ .
            \end{itemize}
        \item[]
        \item A Custom Recurrence $[k=3$, $\vec{c} = (2, 1, 1)].$
            \begin{itemize}
            \item[]
                \item Initial nonzero terms:
                    \begin{align*}
                        X_1 &\ =\  1, \\
                        X_2 &\ =\  2 \cdot 1 + 1\ =\ 3, \\
                        X_3 &\ =\  2 \cdot 3 + 1 \cdot 1 + 1\ =\ 8.
                    \end{align*}
                \item Recurrence for $n > 3$:
                \[ X_n\ =\ 2X_{n-1} + X_{n-2} + X_{n-3}. \]
                \item Sequence: $1, 3, 8, 20, 51, 130, \dots$\ .
            \end{itemize}
    \end{itemize}
\end{example}

 We generalize the concept of a $\vtr{c}$-recursive sequence to higher dimensional vectors. To do so, we generalize the notion of \textbf{\emph{$k$-bonacci vectors}} from \textcolor{blue}{\cite{ABJ11}}.

% Credit: Tianyu created, Alan revised, Sebastian edited, Tianyu edited
\begin{definition}\label{Def: Rec Vectors}
    We define the \textbf{$\vtr{c}$-recurrence vector sequence}, $(\vtr{X_n})_{n\in \mathbb{Z}} \subseteq \mathbb{Z}^{k-1}$,  as
    \begin{enumerate}
        \item[(1)] $\vtr{X}_0 := \vtr{0}$,
        \item[(2)] $\vtr{X}_{-i} := \vtr{e}_i$ for all $1 \leq i \leq k - 1$ where $\vtr{e}_i$ is the standard basis vector, and
        \item[(3)] $\X_n := c_1\X_{n - 1} + c_2\X_{n - 2} + \cdots + c_k\X_{n-k}$  for all  $n \in \mathbb{N}$.
    \end{enumerate}
\end{definition}

Since $c_k \neq 0$, we see that after rearranging terms and shifting indices that
\begin{equation} \label{eq:1}
     \vtr{X}_n\ =\ \frac{\vtr{X}_{n+k} - \sum_{i =1}^{k - 1}c_i\vtr{X}_{n + k - i}}{c_k},
\end{equation}
which lets us work backwards to define vectors with negative indices. As we assume $c_k = 1$, Equation \textcolor{blue}{\eqref{eq:1}} can be simplified to
\begin{equation} \label{eq:2}
     \vtr{\mathbf{X}}_n\ =\ \vtr{\mathbf{X}}_{n+k} - \sum_{i=1}^{k-1}c_i\vtr{\mathbf{X}}_{n+k-i}.
\end{equation}

 \begin{remark}\label{rem: c_k=1}
     Since Equation \textcolor{blue}{\eqref{eq:2}} lets us define $\X_{-n}$ recursively for all $n \geq k$ we can guarantee that for each $n\in \mathbb{Z}$, $\vtr{X}_n$ is well defined as an element of $\mathbb{Z}^{k-1}$. However, if $c_k \neq 1$ it is not guaranteed that for every $n<0$ all terms of $\vtr{X}_{n}$ are integers\footnote{It is not necessary to restrict $c_k=1$ and instead we might ask that $c_k \mid c_i$ for all $1\leq i \leq k-1$ to guarantee integer coefficients in Equation \textcolor{blue}{\eqref{eq:1}}. Nevertheless, restricting $c_k=1$ will be necessary in establishing our main results.}.
 \end{remark}

From this point forward $( \vtr{X}_n )_{ n \in \mathbb{Z}}$ represents the $\vtr{c}$-recurrence vector sequence.

% Credit: Tianyu
\begin{example}\label{Ex: (2,1,1) Rec Vectors}
    Here, we list several terms of the $\mathbf{c}$-recurrence vector sequence with $c_1 = 2, c_2 = 1, c_3 = 1$.
\begin{align*}
    \vtr{\mathbf{X}}_0  \ &\ =\  (0,0) \\
    \vtr{\mathbf{X}}_{-1} &\ =\ (1,0) \\
    \vtr{\mathbf{X}}_{-2} &\ =\ (0,1) \\
    \vtr{\mathbf{X}}_{-3} &\ =\  (-2,-1) \\
    \vtr{\mathbf{X}}_{-4} &\ =\ (3,-1) \\
    \vtr{\mathbf{X}}_{-5} &\ =\ (1,4) \\
    \vtr{\mathbf{X}}_{-6} &\ =\ (-9,-3) \\
    \vtr{\mathbf{X}}_{-7} &\ =\ (10,-6) \\
    \vtr{\mathbf{X}}_{-8} &\ =\ (9,16) \\
    \vtr{\mathbf{X}}_{-9} &\ =\ (-38,-7).
\end{align*}
\end{example}

Note how the recurrence $\vtr{\mathbf{X}}_n = \vtr{\mathbf{X}}_{n+k} - \sum_{i=1}^{k-1}c_i\vtr{\mathbf{X}}_{n+k-i}$ from Equation \textcolor{blue}{\eqref{eq:2}} generates these vectors backward.\\

Zeckendorf \textcolor{blue}{\cite{Zeckendorf72}} proved that any positive integer $n$ can be written uniquely as a sum

$$n\ =\ \sum_{i\geq 1}d_iF_i$$

\noindent such that $d_i\in \{0,1\}$ for all $i$, and no string of two consecutive $d_i$'s equal $1$. Equivalently, for every $n\in \mathbb{N}$ there is a unique infinite string of nonnegative integers $(d_1d_2,\ldots)$ with finitely many nonzero terms such that $n=\sum_{i\geq 1}d_iF_i$, no $d_i$ exceeds $1$, and no copy of $11$ can be found in the string. This result was greatly extended to the more general realm of PLRS's in \textcolor{blue}{\cite{MW12}} by using similar restrictions on the underlying string of coefficients, as well as in many different settings  (for example \textcolor{blue}{\cite{ABJ11}}, \textcolor{blue}{\cite{CHHMPV18b}} and \textcolor{blue}{\cite{KKMW11}}). Informally, uniqueness is ensured by restricting both the size of each term and forbidding a copy of the defining PLRS recurrence within this underlying string of coefficients. Such a representation of an integer $N$ is called a \textit{legal decomposition} or \textit{legal sum} of $N$ for the PLRS $X_n$. We exactly match these restrictions in the vector case with the following definition.

% Credit: Sebastian
\begin{definition}
\label{defintion:SR}
    Let $\vtr{v} \in \mathbb{Z}^{k-1}$ be any vector. We call a sequence of nonnegative integers $(a_n)_{n = 1}^{\infty}$ a $\vtr{c}$ \textbf{\textit{- satisfying representation}} or simply a \textbf{$\vtr{c}$-SR of $\vtr{v}$ }if either $\vtr{v}=\vtr{0}$ and $a_n=0$ for all $n\in \mathbb{N}$ or if the sequence $(a_n)_{n=1}^\infty$ satisfies the following conditions.
    \begin{enumerate}
        \item  There exists an $m\in \mathbb{N}$ such that $a_n=0$ for all $n>m$.
        \item We have $\vtr{v} = \sum_{n = 1}^m a_n \vtr{X}_{-n}$.
        \item We have $a_m > 0$, and $a_n \geq 0$ for all $1\leq n \leq m$.
        \item One of the following holds. \begin{itemize}
                \item We have $m < k$ and $a_n = c_n$ for all $1 \leq n \leq m$.
                \item There exists $s \in \{1, \ \ldots, \ k\}$ such that
                \begin{equation}
                    a_1\ =\ c_1,\hspace{0.2cm} \ldots, \hspace{0.2cm} a_{s-1}\ =\ c_{s-1} \hspace{0.3cm} \text{and }\hspace{0.2cm} a_s\ <\ c_s,
                \end{equation}
                 and there exists an $ \ell \geq 0$ such that $a_{s+1},\ \ldots, \ a_{s+\ell} = 0$, and $(a_{s+ \ell +n} )_{n=1}^{\infty}$ is $\vtr{c}$-SR.
            \end{itemize}
    \end{enumerate}
    We also refer to a finite sequence of nonnegative integers, $(a_n)_{n=1}^m$, as a $\vtr{c}$-SR if the sequence $(b_n)_{n=1}^\infty$ defined by

    \[b_n\ =\ \begin{cases}
        a_n, & \text{if }1\leq n\leq m; \\
        0, & \text{otherwise}
    \end{cases}\]

    \noindent is a $\vtr{c}$-SR.

\end{definition}

For example, if $\vtr{c}=(4,2,1)$ then the finite sequence $2,4,2,0,1$ is a $\vtr{c}$-SR of the induced vector while this is not the case for either of the finite sequences $2,4,2,1$ (this sequence contains a copy of $4,2,1$) or $2,4,3$ (as the element $3$ is too large).

For each $k\geq 2$, let $\vtr{1}_k\in \mathbb{Z}^{k-1}$ denote the $(k-1)$-vector consisting of all $1$'s. In our notation, Anderson and Bicknell-Johnson show in \textcolor{blue}{\cite{ABJ11}} that there is a unique vectorized analogue for the $k$-Fibonacci Zeckendorf Theorem.

\begin{theorem}{\rm{(}\textcolor{blue}{\cite{ABJ11}\label{thm: ABJ}}, [\emph{Theorem 2}]\rm{)}}
Every $\vtr{v}\in \mathbb{Z}^{k-1}$ has a unique $\vtr{1}_k$-satisfying representation.

\end{theorem}

\begin{definition}\label{Def: weakly dec}
    We call a vector $\vtr{c}=(c_1,c_2,\ \ldots,\ c_k)$  \textit{\textbf{weakly decreasing}} if for each $1\leq n\leq k-1$ we have $c_n\geq c_{n+1}$.
\end{definition}

We are now ready to state our main result, which generalizes Theorem \textcolor{blue}{\ref{thm: ABJ}} to weakly decreasing $\vtr{c}$-recurrence vectors.

\begin{theorem}\label{thm: main}
    If $\vtr{c}=(c_1,c_2,\ \ldots, \ c_k)$ is weakly decreasing and $c_k=1$ then every $\vtr{v}\in \mathbb{Z}^{k-1}$ has a unique $\vtr{c}$-satisfying representation.
\end{theorem}

% Construction of CHUNKS - Sebastian
 To prove this result, we decompose terms of strings \textit{``close''} to $\vtr{c}$-SR's into groups and process them one by one. Thankfully, the definition of a $\vtr{c}$-SR naturally groups coefficient terms into \textit{``chunks''} that can be separately examined as $\vtr{c}$-SR's themselves. Indeed, the definition ensures that when reading the string of coefficients from left to right one has distinct ``\textit{partially completed}'' copies of $\vtr{c}$ followed by zeros. We formalize this notion of \textit{chunks} in the next definition.

\begin{definition}\label{Def: Chunks}
    Suppose that $a:=(a_n)_{n=1}^\infty$ is a $\vtr{c}$-SR and $m\in \mathbb{N}$ is the largest index such that $a_m>0$.
    Let $n_1:=1$; this is the \textbf{\textit{first element of the first chunk of $a$}}.\\

    \noindent As $(a_n)_{n=n_1}^\infty$ is a $\vtr{c}$-SR, there exists $s_1 \in \{1,\ \ldots, \ k\}$ such that
    \[
    a_{n_1}\ =\ c_1,\ \ldots,\ a_{n_1 +(s_1-1) -1}\ =\ c_{s_1-1}, \text{ and } a_{n_1+s_1-1}\ <\ c_{s_1}.
    \]
    Let $A_1 := \{r \geq n_1+s_1 : a_r \neq 0 \}$. If $A_1$ is a nonempty set we define the \textbf{\textit{first element of the second chunk}} as $n_2:=\min A_1$. Note that, by the definition of a $\vtr{c}$-SR, $(a_n)_{n=n_2}^\infty$ is a $\vtr{c}$-SR as well.\\

    \noindent Recursively define $A_i:=\{r\geq n_i + s_i : a_r \neq 0 \}$. If this is a nonempty set, \textbf{\textit{the first element of the $(i+1)$\textsuperscript{\rm st} chunk}} is defined by $n_{i+1}:=\min A_i$.\\

    As $s_i \geq 1$ it follows that $n_{i+1} \geq n_i+s_i > n_i$ which implies that the $n_i$'s are distinct. Hence, as it is clear that each $n_i\leq m$, we can ensure that this process must terminate and only a finite number of $n_i$'s arise, say $(n_i)_{i=1}^\ell$.\\

    After this process completes we define $\ell$ to be the \textit{\textbf{number of chunks}} of the representation, and we denote it by \emph{$CH(a)$}. Let $n_{\ell+1}=m+1$, and for each $1\leq i\leq \ell$ refer to the string $a_{n_i}a_{n_i+1}\dots  a_{n_{i+1}-1}$ as the \textbf{$i$\textsuperscript{\rm th} \textit{chunk} of $a$}.

\end{definition}

Following the ideas of \textcolor{blue}{\cite{ABJ11}}, to establish the existence of $\vtr{c}$-SR for each $\vtr{v}\in \Z^{k-1}$ we manipulate representations built by adding $1$ to a single coefficient of a $\vtr{c}$-SR.

% Credit: Sebastian
\begin{definition}
\label{defintion:NSR}
    A \textbf{$\vtr{c}$-nearly satisfying representation ($\vtr{c}$-NSR) for $\vtr{v} \in \mathbb{Z}^{k-1}$} is a sequence of nonnegative integers $(a_n)_{n=1}^\infty$ such that $\vtr{v}=\sum_{n=1}^\infty a_n \X_{-n}$ and where there exists an integer $i\in \N $ for which the following hold.
        \begin{itemize}
            \item The sequence $(a_n)_{n=1}^\infty$ is not a $\vtr{c}$-SR.
            \item The sequence $(b_n)_{n=1}^\infty$ defined by

            $$b_n\ =\ \begin{cases}
                a_n-1, & \text{if }n=i; \\
                a_n, & \text{otherwise}
            \end{cases}$$

            \noindent is a $\vtr{c}$-SR.
        \end{itemize}

\end{definition}

As in the case of $\vtr{c}$-SR's we refer to finite sequences $(a_n)_{n=1}^m$ as a $\vtr{c}$-NSR if the sequence $(b_n)_{n=1}^\infty$ defined by

    \[b_n\ :=\ \begin{cases}
        a_n ,& \text{if }1\leq n\leq m; \\
        0, & \text{otherwise}
    \end{cases}\]

    \noindent is a $\vtr{c}$-NSR. If $a=(a_n)_{n=1}^\infty$ is not a $\vtr{c}$-SR (including but not limited to the case where $a$ is a $\vtr{c}$-NSR), define
    \begin{align*}
    I(a)\ &: =\ \begin{cases}
        \max\{j :(a_n)_{n=1}^{j-1}\text{ is a $\vtr{c}$-SR} \}, & \text{if } (a_n)_{n=1}^1\ \text{is a }\vtr{c}\text{-SR};\\
        1, & \text{otherwise}
    \end{cases}
    \end{align*}

    \noindent and call $I(a)$ \textbf{\textit{the first overfilled element of $a$}}. A $\vtr{c}$-NSR, say $(a_n)_{n=1}^m$, is \textit{\textbf{end complete}} if $(b_n)_{n=1}^m$ defined by

     \[b_n\ =\ \begin{cases}
        a_n, & \text{if }1\leq n\leq m-1; \\
        a_n-1, & \text{if }n=m
    \end{cases}\]
 is a $\vtr{c}$-SR with $\ell$ chunks such that the $\ell^{\rm th}$ chunk takes the form $b_{n_\ell}b_{n_{\ell}+1}\dots b_{n_\ell+k-2}$ and where

    $$b_{n_\ell}\ =\ c_1,\ b_{n_{\ell}+1}\ =\ c_2,\ \ldots,\ b_{n_\ell+k-2}\ =\ c_{k-1}.$$

To prove our main result we manipulate a given $\vtr{c}$-NSR of a vector $\vtr{v}$ using operations that, when each is performed, give a representation of $\vtr{v}$.

% Credit: Sebastian created & wrote formal def, Alan typed the informal-simplified definition

\begin{definition}\label{definition:carrying}
    Let $\vtr{v} \in \mathbb{Z}^{k-1} $ be any vector. We call a \textbf{$\vtr{c}$-representation of $\vtr{v}$} a sequence $(a_n)_{n=1}^\infty$ such that $\vtr{v} = \sum_{n=1}^\infty a_n \X_{-n}$ and where there exists an $m\in \N$ such that $a_n=0$ for all $n>m$. Given $i\in \N$ and a $\vtr{c}$-representation of $\vtr{v}$, say $(a_n)_{n=1}^\infty$, we define two processes for obtaining new $\vtr{c}$-representation of $\vtr{v}$.
\begin{itemize}
    \item \textbf{Carrying into $a_i$} gives us a new sequence $(b_n)_{n=1}^\infty$ defined by
    \[
    b_j\ =\
    \begin{cases}
    a_j, & \text{if } j < i \text{  or  } j > i+k; \\
    a_i + 1, & \text{if } j = i; \\
    a_j - c_l, & \text{if } i < j \leq i+k \text{ and } j = i+l.
    \end{cases}
    \]

    \item \textbf{Borrowing from $a_i$} gives us a new sequence $\{d_n\}_{n=1}^\infty$ defined by
    \[
    d_n\ =\
    \begin{cases}
        a_n, & \text{if } n < i \text{ or } n > i+k;\\
        a_i-1, & \text{if } i = n; \\
        a_n + c_l, & \text {if } i < n \leq i+k \text{ and } n = i+l.
    \end{cases}
    \]
\end{itemize}

Note that, due to the $\vtr{c}$-recurrence, in either case $ \vtr{v}  = \sum_{n=1}^\infty b_n \vtr{X}_{-n} = \sum_{n=1}^\infty d_n \vtr{X}_{-n}$.
\end{definition}

More informally,
    \begin{enumerate}
        \item[(1)] \emph{carrying into} $a_i$ increments $a_i$ by 1 and decrements $a_{i +j}$ by $c_j$ for each $j = 1, 2, \ \ldots, \ k$; and
        \item[(2)] \emph{borrowing from} $a_i$ decrements $a_i$ by 1 and increments $a_{i + j}$ by $c_j$ for each $j = 1, 2, \ \ldots, \ k$.
    \end{enumerate}

\begin{table}[htbp]
  \centering
  \vspace{.5cm}
  \begin{tabular}{|>{\centering\arraybackslash}m{0.8cm}|>{\centering\arraybackslash}m{2.7cm}|>{\centering\arraybackslash}m{6cm}|>{\centering\arraybackslash}m{1.3cm}|}
    \hline
    \textbf{Step} & \textbf{Operation} & \textbf{Representation} & \textbf{Vector} \\
    \hline
    1 & Initial $\vtr{c}$-SR & $2\cdot\vtr X_{-2} + \vtr X_{-3}$ & $(-2,1)$ \\
    \hline
    2 & Add $\vtr X_{-3}$ & $2\cdot\vtr X_{-2} + 2\cdot\vtr X_{-3}$ & $(-4,0)$ \\
    \hline
    3 & Borrow from $a_3$ & $2\cdot\vtr X_{-2} + \vtr X_{-3} + 2\cdot\vtr X_{-4} + \vtr X_{-5} + \vtr X_{-6}$ & $(-4,0)$ \\
    \hline
    4 & Carry into $a_1$ & $\vtr X_{-1} + \vtr X_{-4} + \vtr X_{-5} + \vtr X_{-6}$ & $(-4,0)$ \\
    \hline
  \end{tabular}
  \caption{Illustration of how carrying/borrowing operations can generate a $\vtr{c}$-SR from a $\vtr{c}$-NSR with recurrence $\vtr X_n = 2\vtr X_{n-1} + \vtr X_{n-2} + \vtr X_{n-3}$. }\label{tab:carry-borrow}

\end{table}

To avoid negative coefficients, we only carry into $a_i$ when $a_{i+j} \geq c_j$ for each $j = 1, 2,\ \ldots,\ k$ and we only borrow from strictly positive $a_i$. Table \textcolor{blue}{\ref{tab:carry-borrow}} shows how carrying and borrowing can be quite useful in going from a $\vtr{c}$-NSR representation of $\vtr{v}$ to a $\vtr{c}$-SR representation of $\vtr{v}$.

\begin{remark}\label{rem:index0}
    Since $\vtr{X}_0 = \vtr{0}$, the coefficient of $\vtr{X}_0$ does not affect the vector that a sequence represents.
    Hence any $\vtr{c}$-representation $(a_n)_{n=1}^{\infty}$ of $\vtr{v}$ extends to a sequence $(a_n)_{n=0}^{\infty}$ that still represents $\vtr{v}$ by setting $a_0=0$, and carrying into $a_0$,
    given by the formulas of Definition~\textcolor{blue}{\ref{definition:carrying}} with $i=0$, is well defined.
    Carrying into $a_0$ also preserves $\vtr{v}$, since it subtracts $c_1\vtr{X}_{-1}+\cdots+c_k\vtr{X}_{-k}$, which equals $\vtr{X}_0=\vtr{0}$ by Equation~\textcolor{blue}{\eqref{eq:2}}.
\end{remark}

%Phrasing
 In proving our main result we prove that a $\vtr{c}$-NSR can be transformed into $\vtr{c}$-SR by a finite number of borrowing and carrying operations. In showing that our underlying algorithm that accomplishes this eventually terminates, it is useful to define functions that count the sum of all or of a subset of the coefficients on an eventually zero, infinite string.
% Credit: Tianyu

\begin{definition}\label{Def: sums of coefs}
    For each sequence of nonnegative integers $a:=(a_n)_{n=1}^\infty$ where there exists an $m\in \N$ such that $a_{n}=0$ for all $n>m$, we define $G(a)$ as the sum of elements of $a$,
\begin{equation}
    G(a)\ :=\ \sum_{i=1}^{m}a_i,
\end{equation}
and for each $n\in \N$ define $G_n(a)$ as the sum of every term of index less than $n$,

\begin{equation}
    G_n(a)\ :=\ \sum_{i=1}^{n-1}a_i.
\end{equation}

\end{definition}

\begin{example}{\label{Ex: Sum of ceof}}
    Given the string $a=21012100 \dots$, we have $G(a)=2+1+0+1+2+1+0=7$ and $G_4(a)=2+1+0=3$.
\end{example}
%\newpage
\begin{remark}\label{Rem:sumofcoefs bor/car}
    Note that for a $\vtr{c}$-representation of $\vtr{v}$, say $a:=(a_n)_{n=1}^\infty$, with $G(a)=R$, if $b$ and $d$ are the $\vtr{c}$-representations of $\vtr{v}$ resulting from carrying into some $a_n$ and borrowing from some $a_n$ respectively, then $G(b)=R - \sum_{i=1}^k c_i + 1$ and $G(d)=R + \sum_{i=1}^kc_i -1$. Note that if $b$ is the result of carrying into $a_0$ as discussed in Remark \textcolor{blue}{\ref{rem:index0}}, then $G(b)=R - \sum_{i=1}^k c_i$
\end{remark}

\section{Proofs of Main Results}

There is a natural map that transforms linear combinations of truncated $\vtr{c}$-recurrence vectors into the underlying $\vtr{c}$-recurrence sequences. In the following definition and the following lemma, recall that $X_n=0$ for $-k+2\leq n\leq 0$.

% Credit: Erica created, Alan edited, Tianyu modified
\begin{definition}\label{def:Sn}
    For $n\geq 1,S_n:\mathbb{Z}^{k-1}\xrightarrow{}[0,X_n)$ is the scalar product defined by
    \[S_n(\vtr{v})\ =\ \vtr{v}\cdot(X_{n-1},\ \ldots,\ X_{n-k+1})\pmod{X_n}\]
    where each $X_n$ is defined as in Definition 1.1.
\end{definition}

% Credit: Erica created, Alan edited and rewrote proof, Tianyu edited and rewrote proof

The following lemma follows by the same argument as in \textcolor{blue}{\cite{ABJ11}}. For the sake of completeness and to highlight our specific case, its proof is included below.

\begin{lemma}
\label{lem:S_n}
  For $n\geq 1$ we have  $S_n(\sum_{i=1}^{n-1}a_i\vtr{X}_{-i}) = \sum _{i=1}^{n-1}a_iX_{n-i} \pmod{(X_{n})}$.
\end{lemma}

% Credit: Tianyu
\begin{proof}
    When $i=0$ we have 

    \[S_n(\X_{-0})\ =\ X_{n }\pmod{X_n}=0\]
    
   \noindent and for $1 \leq i \leq k - 1$, we have $\vtr{X}_{-i} = \vtr{e}_i$, by definition. Therefore,
    \[S_n(\X_{-i})\ =\ X_{n - i}\pmod{X_n}.\]

    \noindent For $k\leq i\leq n-1$, using Equation \textcolor{blue}{\eqref{eq:2}} and supposing inductively that $S_n(\vtr{X}_{-j})=X_{n-j}$ for each $0\leq j\leq i-1$ allows us to conclude that
    \begin{align*}
        S_n(\X_{-i}) &\ =\ \X_{-i}\cdot(X_{n-1},\ \ldots,\ X_{n-k+1})\pmod{X_n} \\
                     &\ =\ \left(\vtr{X}_{-i+k} - \sum_{j =1}^{k - 1}c_j\vtr{X}_{-i+k- j}\right)\cdot(X_{n-1},\ \ldots,\ X_{n-k+1})\pmod{X_n} \\
                     &\ =\ \left(X_{n-i+k} - \sum_{j =1}^{k - 1}c_jX_{n-i+k- j}\right)\pmod{X_n} \\
                     &\ =\ X_{n-i}\pmod{X_n}.
    \end{align*}

    Therefore, for each $1 \leq i \leq n-1$ we have that $S_n(\X_{-i}) = X_{n - i} \pmod{X_n}$. By linearity of $S_n$, the proof is complete.
\end{proof}

% credit sebastian
\begin{proposition}\label{prop:End complete}
    Each end complete $\vtr{c}$ -NSR can be transformed into a $\vtr{c}$-SR by a finite sequence of carrying operations. Furthermore, by Remark \textcolor{blue}{\ref{Rem:sumofcoefs bor/car}}, this process only reduces the sum of the coefficients.
\end{proposition}
\begin{proof}
     We proceed by induction over the number of chunks of the representation. Let $a:=(a_n)_{n=1}^m$ be an end complete $\vtr{c}$-NSR with only one chunk. By Remark~\textcolor{blue}{\ref{rem:index0}}, we may extend $a$ to $(a_n)_{n=0}^m$ by setting $a_0=0$, and carry into $a_0$ without changing the vector it represents. In this case, note that $m = k$ and
    \[
      a_1\ =\ c_1 , \ \ldots ,\ a_k\ =\ c_k=1.
    \]
    Then, by carrying into $a_0$, we obtain the sequence $(b_n)_{n=0}^m$, where $b_0=1$ and $b_i=0$ for all $1\leq i\leq m$. Note that $(b_n)_{n=1}^m$ is the zero vector which is, indeed, a $\vtr{c}$-SR.\\

    Suppose that all end complete $\vtr{c}$-NSR's with $\ell$ chunks can be transformed into a $\vtr{c}$-SR. Let $(a_n)_{n=1}^m$ be an end complete $\vtr{c}$-NSR with $\ell+1$ chunks. For each $1\leq i\leq \ell+1$ denote $a_{n_i}$ as the first coefficient of the $i^{\rm th}$ chunk as in Definition \textcolor{blue}{\ref{Def: Chunks}}. By assumption we can carry into $a_{n_{\ell+1} -1}$. Let $b$ be the resulting sequence of such a carrying operation; then, we have exactly two cases to consider.
            \begin{description}
                \item[Case 1:] $b$ is a $\vtr{c}$-SR and we are done, or\\

                \item[Case 2:] $b$ is an end complete $\vtr{c}$-NSR and the desired result follows by induction.
            \end{description}

     \noindent Indeed, suppose that Case 1 does not hold; then by Definition \textcolor{blue}{\ref{defintion:SR}} there must exist $s\in \{1, \ \ldots, k \} $ such that
     $$a_{n_\ell} \ =\ c_1, \ \ldots, \ a_{n_\ell -1 + s -1} \ = \ c_{s-1}, \  a_{n_\ell -1 + s} < c_s $$
 and there exists $p \geq 0$ such that $a_{n_\ell -1 +s + i }  = 0$, for all $i\leq p$, and $n_\ell -1 + s + p  =  n_{\ell +1}-1$. Notice that by Definition \textcolor{blue}{\ref{definition:carrying}}, and since we assumed $a$ had $\ell+1$ chunks, $ b_{q}  =  0$ for all $q\geq n_{\ell +1}$. If $p=0$ then $n_{\ell+1} - 1 = n_\ell + s-1$, and so $b_{n_\ell -1 + s} \leq c_s$; otherwise $b$ is exactly
     %still the spacing, same as above
    \[
        (a_1,\ \ldots, \ a_{n_\ell}\ =\ c_1,\ a_{n_\ell +1}\ =\ c_2,\ \ldots,\ a_{n_\ell-1  +s}\ <\ c_s, 0, \ \ldots,\ 0,\  b_{n_{\ell+1} -1} = 1)
    \]
    which is clearly a $\vec{c}$-SR.
     Hence, the only possible way for $b$ not to be a $\vec{c}$-SR is that  $s = k$ and $b_{n_\ell -1 + s} = c_k$, in other words, $b$ is end complete. As the carrying operation that defines $b$ only affects the $\ell$-th chunk of $a$, $b$ is an end complete $\vtr{c}$-NSR with $\ell$ chunks and so the result follows by induction on the number of chunks.
\end{proof}

We now prove that weakly decreasing $\vtr{c}$-NSR's can be transformed into satisfying representations. We note that the algorithm utilized in the proof of this result is similar to the algorithm from the appendix of \textcolor{blue}{\cite{KKMW11}} that converts any decomposition of an integer over a PLRS into a legal decomposition. However, as we define the vectors using a backwards recursion, it is not the case that we could eventually borrow into zero terms. In fact, it is very possible in the case where $\vtr{c}$ is not weakly decreasing for this process to never terminate (see Example \textcolor{blue}{\ref{Ex: Bad Algorthm}}).

\begin{proposition}\label{prop: general cSR}
Every $\vtr{c}$-NSR, where $\vtr{c}$ is a weakly decreasing vector and $c_k = 1$, can be transformed into a $\vtr{c}$-SR by a finite number of borrowing and carrying operations. In addition, each of these transformations does not increase the sum of the coefficients of the resulting representation.
\end{proposition}
\begin{proof}

Recall that if $\vtr{c}$ is a weakly decreasing vector then $c_{i+1} \leq c_i$ for all $i \in \{1,\ \ldots,\ k-1 \}$.
Throughout this proof,  the symbol $a^{i}$ represents the $i$\textsuperscript{th} iteration of a process that consists of a finite number of borrows and carries.
Suppose that $a^{0}:=( a_n^{0} )_{n=1}^{\infty}$ is a $\vtr{c}$-NSR and let $c:=\sum_{i=1}^k c_i$.

By Definition \textcolor{blue}{\ref{defintion:NSR}} there exist $ p_0\geq 1$ and $0\leq j_0< k$ such that $I(a^0)=n_{p_0}+j_0$ and where
\begin{align}
    & a_{n_{p_0}}\ =\ c_1,\ \ldots,\ a_{n_{p_0} + j_0 -1}\ =\ c_{j_0},\  a_{n_{p_0}  + j_0}\ \geq \ c_{j_0+1}. \label{eq:3}
\end{align}

Since $n_{p_0}+j_0$ is the first overfilled element of $a$, we have that $(a_n^0)_{n=1}^{n_{p_0} + j_0-1}$ is a $\vec{c}$-SR. We now proceed by case analysis to define $a^1$.\\

    \noindent \textbf{Case 1: $j_0=k-1$.} Then the sequence $(\alpha_n^0)_{n=1}^{n_{p_0}+j_0}$ defined by
    
    $$\alpha_n^0=\begin{cases}
        a_n^0, & 1\leq n\leq n_{p_0}+j_0-1;\\
        1, & n=n_{p_0}+j_0
    \end{cases}$$
    
    \noindent is end complete and so by applying Proposition \textcolor{blue}{\ref{prop:End complete}} to only this finite string, a finite sequence of carries transforms $(\alpha_n^0)_{n=1}^{n_{p_0} + j_0}$ into a $\vtr{c}$-SR that, due to carrying, can be defined over one less index, say $(b^{0}_n)_{n=1}^{n_{p_0} + j_0-1}$ since $\alpha_{n_{p_0}+j_0}^0$ becomes zero after the carry. Define $a^1$ by

    $$a_n^{1}\ =\ \begin{cases}
        b^0_n, & 1\leq n\leq n_{p_0}+j_0-1;\\
        a_n^0-1, & n= n_{p_0}+j_0;\\
        a_n^0, & \text{otherwise.}
    \end{cases}$$

    \noindent In this case, note that $a^1$ is the result of a finite number of carries of $a^0$ and so we have that $G(a^1)\leq G(a^0)-(c-1)$.\\

   \noindent \textbf{Case 2: $0\leq j_0<k-1$.} Note that in this case, $a_{n_{p_0}  + j_0}\ >  \ c_{j_0+1}$. Borrow from $a_{n_{p_0}+j_0}^0$ to define $d^0$. In particular, define

    $$d^0_n\ =\ \begin{cases}
        a_n^0+c_i, & \text{if } n=n_{p_0}+j_0+i\ \text{for some }1\leq i\leq k;\\
        a_n^0-1, & \text{if } n=n_{p_0}+j_0;\\
        a_n^0 & \text{otherwise.}
    \end{cases}$$

  \noindent Since $\vtr{c}$ is weakly decreasing, then for all $1\leq  i\leq k-j_0-1$
\begin{align*}
    &d^0_{n_{p_0} +j_0 +i}\ \geq\ c_{i}\ \geq\ c_{j_0+i+1}
 \end{align*}
and, thus,

\begin{align}
    &d^0_{n_{p_0} +j}\ \geq\ c_{j+1}\ \text{for each }0\leq j\leq k-1.
 \end{align}

\noindent Moreover, if we define $(e^0_n)_{n=1}^{n_{p_0} + k-1}$ by

$$e^0_n\ =\ \begin{cases}
        d^0_n, & \text{if  }1\leq n< n_{p_0}+j_0\ \text{ or }\ n_{p_0}+k\leq n;\\
        c_{i+1}, & \text{if  } n=n_{p_0}+i\ \text{ for some }j_0\leq i\leq k-1,
    \end{cases}$$

\noindent then the finite string $(e^0_n)_{n=1}^{n_{p_0} + k-1}$ is end complete and so by applying Proposition \textcolor{blue}{\ref{prop:End complete}}, a finite sequence of carries transforms $(e^0_n)_{n=1}^{n_{p_0} + k-1}$ into a $\vtr{c}$-SR, say $f^0:=(f^0_n)_{n=1}^{n_{p_0} + k-1}$. Define $a^1$ by

    $$a_n^{1}\ =\ \begin{cases}
        f^0_n, & 1\leq n< n_{p_0}+j_0;\\
        d^0_n-c_{i+1}, & \text{if } n=n_{p_0}+i\ \text{ for some }j_0\leq i\leq k-1;\\        d^0_n, & \text{otherwise.}
    \end{cases}$$

    \noindent In particular, $a^1$ is obtained from $a^0$ by a single borrow and then at least one carry. By Remark \textcolor{blue}{\ref{Rem:sumofcoefs bor/car}} and with consideration for the possibility for carrying into $a_0$ as discussed in Remark \textcolor{blue}{\ref{rem:index0}}, this implies that $G(a^1)\leq G(a^0)$.\\

If $a^1$, defined in either case, is a $\vtr{c}$-SR, then we are done. If not, then as above there exist $p_1\geq 1$ and $0\leq j_1< k$ such that $I(a^1)=n_{p_1}+j_1$ and where
\begin{align}
    & a_{n_{p_1}}\ =\ c_1,\ \ldots,\ a_{n_{p_1} + j_1 -1}\ =\ c_{j_1},\  a_{n_{p_1}  + j_1}\ \geq\ c_{j_1+1}. \label{eq:4}
\end{align}

There are two important properties we prove. However, to do so we first define an ordering on sequences. We say that for strings $a$ and $b$ that $a\prec b$ if $a_r< b_r$ where $r$ is the first index such that $a_r\neq b_r$.\\

\noindent \textbf{Claim 1:} Either $G(a^1)<G(a^0)$ or $a^0\prec a^1$.\\

\begin{proof}[Proof of Claim 1.]
    
    Note that by Remark \textcolor{blue}{\ref{Rem:sumofcoefs bor/car}}, if $G(a^1)=G(a^0)$ then $a^1$ is obtained from $a^0$ by exactly one borrow from $a^0_{n_{p_0}+j_0}$ and then one carry into $a^{0}_{n_{p_0}-1}$ where $n_{p_0}-1\geq 1$. Thus, as $a^{0}_{n_{p_0}-1}+1=a^{1}_{n_{p_0}-1}$ and no earlier indices are changed, we have that $a^0\prec a^1$ as desired.
\end{proof}

Since $n_{p_1}+j_1$ is the first overfilled element of $a^1$, we have that $(a_n^1)_{n=1}^{n_{p_1} + j_1-1}$ is a $\vec{c}$-SR.

Inductively, if $a^{\ell}$ is not a $\vtr{c}$-SR and is defined for some $\ell\geq 1$, then we may define $n_{p_{\ell}}$ and $j_{\ell}$ such that $n_{p_\ell}+j_\ell$ is the first overfilled element of $a^\ell$. Then we have that $(a_n^\ell)_{n=1}^{n_{p_\ell} + j_\ell-1}$ is a $\vec{c}$-SR and employ the same case analysis as above to construct $a^{\ell+1}$. If $a^{\ell+1}$ is a $\vtr{c}$-SR then we are done. Otherwise, we may define $n_{p_{\ell+1}}$ and $j_{\ell+1}$ such that $n_{p_{\ell+1}}+j_{\ell+1}$ is the first overfilled element of $a^{\ell+1}$ and have that $(a_n^{\ell+1})_{n=1}^{n_{p_{\ell+1}} + j_{\ell+1}-1}$ is a $\vec{c}$-SR. Furthermore, by applying Claim 1, either $G(a^{\ell+1})<G(a^{\ell})$ or $a^\ell\prec a^{\ell+1}$.\\

\noindent \textbf{Claim 2:} There exists a $q\in \N$ such that $a^{q}$ is a $\vtr{c}$-SR.

\begin{proof}[Proof of Claim 2.]
    Suppose not. Let $Z:=\{i\in \N\ :\ G(a^{i})<G(a^{i-1})\}$. Note that $|Z|\leq G(a^0)$. Indeed, if $|Z|> G(a^0)$ then there exists a subset $\{z_1,\ldots,z_{G(a^0)+1}\}\subset Z$ where $z_{j}<z_{j+1}$ for each $1\leq j\leq G(a^0)$. Then

    \begin{align*}
        G(a^{z_{G(a^0)+1}})\ \leq\ G(a^{z_{G(a^0)}})-1\ \leq\ G(a^{z_{G(a^0)-1}})-2\ &\leq\ \cdots\ \leq\ G(a^{z_1})-G(a^0) <\ 0.
    \end{align*}

    Let $z:=\max\{i:\ i\in Z\cup \{0\}\}$ and define $\gamma:=G(a^{z})$. Note that for all $j>z$ it must be the case that $G(a^{j})=G(a^{j-1})$ and so $a^{j-1}\prec a^{j}$. Thus, we have an infinite chain of strings

    $$a^{z}\ \prec\ a^{z+1}\ \prec\  a^{z+2}\ \prec\ \cdots\ $$

    \noindent each with the sum of coefficients equal to $\gamma$. For any sequence whose terms sum to $\gamma$, say $s$, and where $(s_{n_i})_{i=1}^t$ is the finite substring of nonzero elements of $s$, define $P(s)$ to be the element of $\mathbb{N}^\gamma$ given by 

    $$(\underbrace{n_1,\ n_1,\ \ldots\ , n_1}_{s_{n_1}},\ \underbrace{n_2,\ n_2,\ \ldots,\ n_2,}_{s_{n_2}}\ \ldots,\  \underbrace{n_t,\ n_t,\ \ldots,\  n_t}_{s_{n_t}}).$$

    Note that if $s\prec s'$ then $P(s')<_{\text{lex}}P(s)$ where $<_{\text{lex}}$ denotes the lexicographic ordering of elements in $\mathbb{N}^\gamma$. Hence,

    $$P(a^{z})\ >_{\text{lex}}\ P(a^{z+1})\ >_{\text{lex}}\  P(a^{z+2})\ >_{\text{lex}}\ \cdots\ $$

    is a chain of elements in $\mathbb{N}^\gamma$ that are decreasing infinitely in the lexicographical ordering. As no such infinitely decreasing chain can exist, we have the desired contradiction.

\end{proof}
The proof is complete as Claim 2 establishes that process must terminate in a finite number of iterations to a $\vtr{c}$-SR. Note, in addition, that in every case, this process does not increase the sum of the coefficients as for every iteration $i$, $G(a^i)\leq G(a^{i-1})$.\\
\end{proof}

We now prove our main result.

\begin{proof}[Proof of Theorem \textcolor{blue}{\ref{thm: main}}]

We need to show both existence and uniqueness of the desired representations.\\

\textbf{Existence:} By definition, $\vtr{0}\in \mathbb{Z}^{k-1}$ has a $\vtr{c}$-satisfying representation.  By induction we need only to show that if $\vtr{v}$ has a  $\vtr{c}$-satisfying representation then so does $\vtr{v}+\vtr{X}_{-i}$ for any $1\leq i\leq k$. Indeed, all vectors with nonnegative entries can be obtained by adding linear combinations of the basis vectors $\vtr{X}_{-i}=\vtr{e}_i$ for $1\leq i\leq k-1$ and then as $\vtr{X}_{-k}=-\sum_{i=1}^{k-1}c_i\vtr{e}_i$ with each $c_i\geq 1$, adding multiples of $\vtr{X}_{-k}$ reaches all possible vectors. These inductive steps involve increasing a coefficient in a
$\vtr{c}$-satisfying representation by one. Either the resulting representation is already a $\vtr{c}$-satisfying representation and we are done or it is a $\vtr{c}$-NSR. Hence, existence follows from Proposition \textcolor{blue}{\ref{prop: general cSR}}.\\

\textbf{Uniqueness:}  Suppose that $a$ and $b$ are satisfying representations of a vector $\vtr{v}\in \mathbb{Z}^{k-1}$ and $M\in \mathbb{N}$ is such that $a_i=0=b_i$ for all $i\geq M$. Then by Lemma \textcolor{blue}{\ref{lem:S_n}} and since any legal sequence $a$ for the PLRS $(X_{i})_{i\geq 1}$ satisfies $\sum_{i=1}^{M-1}a_i X_{M-i}<X_M$ we have that

$$S_M(a)\ =\ \sum_{i=1}^{M-1}a_i X_{M-i} \pmod{X_M}\ =\ \sum_{i=1}^{M-1}a_i X_{M-i}$$

\noindent and

$$S_M(b)\ =\ \sum_{i=1}^{M-1}b_i X_{M-i} \pmod{X_M}\ =\ \sum_{i=1}^{M-1}b_i X_{M-i}.$$

\noindent Hence, $S_M(a)=S_M(b)$ implies $a=b$ by the scalar uniqueness of a $\textit{legal}$ decomposition for a PLRS (see \textcolor{blue}{\cite{KKMW11}}, Theorem 1.3). As then, $a=b$, the desired result now follows.

\end{proof}

\begin{remark}\label{Rem: C-allspan}
Even in the case where $\vtr{c}=(c_1,c_2,\ \ldots, \ c_{k-1},1)$ and where the first $k-1$ terms are arbitrary nonnegative entries such that $c_1>0$, any $\vtr{c}$-SR of a vector $\vtr{v}\in \mathbb{Z}^{k-1}$ is unique (if, of course, such a representation for $\vtr{v}$ exists). Indeed, the uniqueness argument within Theorem \textcolor{blue}{\ref{thm: main}} only relies on the fact that $S_n$ gives an injection between $\vtr{c}$-SRs with under $n$ digits and legal decompositions for the associated scalar PLRS.

\end{remark}

\section{Additional Properties of $\vtr{c}$-SRs.}

We extend the probabilistic analysis of Zeckendorf decompositions to the multidimensional setting. By establishing a bijection between vector-valued decompositions and their scalar counterparts, we inherit various probabilistic properties for $\vtr{c}$-SR's.
% This are defintions for theoremos we are aiming to extend

%%%%%%%%%%%%%%%%%%%%%%%%%%%%%%%%%%%%%%%%%%%

% New Defintions for inherit properties
\begin{definition} \label{def: Rn}
    Let $n\in \N$; we define
    \begin{align*}
        R_n^{\vec{c}} \ :=\ \Big\{ \sum_{i=1}^{n} a_i \X_{-i} \ :\ (a_i)_{i=1}^\infty \text{ is a } \vec{c}\text{-SR with } a_n\geq 1 \text{ and } a_i=0 \text{ for all } i>n \Big\},
    \end{align*}
    that is, $R_n^{\vec{c}}$ is the set of vectors whose $\vec{c}$-SR has its last nonzero coefficient at index $n$.
\end{definition}
This definition serves as the natural way of generalizing the interval $[X_n, X_{n+1})$ to higher dimensions; using Remark \textcolor{blue}{\ref{remark:RnD}}, in Figure \textcolor{blue}{\ref{fig: D_10 in (2,1,1) large}} we illustrate $R_{1}^{\vec{c}}$ to $R_{10}^{\vec{c}}$ for  $\vtr{c}=(2,1,1)$. As expected, many properties that hold for decompositions of elements in $[X_n, X_{n+1})$ also hold in the multidimensional analogue.

For clarity, we provide two concrete examples of properties that our multidimensional Zeckendorf representations inherit through the map $S_n$. There are many more statistical and probabilistic properties than those mentioned (for example \textcolor{blue}{\cite{BILMT2}}, \textcolor{blue}{\cite{BILMT13}} and \textcolor{blue}{\cite{BBGILMT13}}).

% Gausiannity
\subsection{Gaussian Convergence}

\begin{definition} (Associated Probability Space to a Positive Linear Recurrence Sequence) \textcolor{blue}{\cite{MW12}}. Let $(H_n)_{n=1}^\infty$ be a PLRS. For each $n$, consider the discrete outcome space
\begin{align*}
    \Omega_n \ := \ \{H_n,\ H_n+1, \ \ldots, \ H_{n+1}-1 \}
\end{align*}
with probability measure
\begin{align*}
    \mathbb{P}_n(A) \ := \ \sum_{\omega \in A} \dfrac{1}{H_{n+1} - H_n}, \ A\subset \Omega_n.
\end{align*}
In addition, define the random variable $K_n$ by setting $K_n(\omega)$ equal to the number of summands of $\omega \in \Omega_n$.
\end{definition}

This definition is extended naturally to higher dimensions in the following way.
\begin{definition}(Inherited Probability Space to a Positive Linear Recurrence Vector Sequence). Let $\vec{c}$ be a weakly decreasing vector with $c_k = 1$. Consider the discrete outcome space
\begin{align*}
    \Omega_n^{\vec{c}} \ := \ R_n^{\vec{c}} ,
\end{align*}
with probability measure
\begin{align*}
    \mathbb{P}_n^{\vec{c}}(B) \ :=\ \sum_{\omega \in B} \dfrac{1}{X_{n+1} - X_n} ,\ B \subset \Omega_n^{ \vec{c} };
\end{align*}
 this is the uniform probability measure on $\Omega_n^{\vec{c}}$, since $|R_n^{\vec{c}}| = X_{n+1}-X_n$ by Lemma \textcolor{blue}{\ref{lem:Sn&Kn}}. We define the random variable $K_n^{\vec{c}}$ by setting $K_n^{\vec{c}}(\omega)$ equal to the number of summands of $\omega \in \Omega_n^{\vec{c}}$; equivalently $K_n^{\vec{c}}(\omega) = G(a)$, where $a$ is the unique $\vec{c}$-SR of $\omega$.
\end{definition}

\begin{lemma} \label{lem:Sn&Kn}
    For each $n\in \N$ and each $j\geq 0$,
    \begin{align*}
        |\{\vec{v}\in R_n^{\vec{c}} \ :\ K_n^{\vec{c}}(\vec{v})=j\}| \ =\ |\{\omega\in [X_n, X_{n+1}) \ :\ K_n(\omega)=j\}|,
    \end{align*}
    where $K_n$ is the random variable of the PLRS $(X_n)_{n=1}^\infty$ of Definition \textcolor{blue}{\ref{def:c-rec seq}}. In particular, $|R_n^{\vec{c}}| = X_{n+1}-X_n$ and $K_n^{\vec{c}}$ has the same distribution as $K_n$.
\end{lemma}
\begin{proof}
    For $m\in \N_0$, let $D_m$ be the set of vectors $\vec{v}\in \Z^{k-1}$ whose $\vec{c}$-SR $(a_i)_{i=1}^\infty$ satisfies $a_i=0$ for all $i>m$ (when $k=3$ this is the set of Definition \textcolor{blue}{\ref{Def: D_n's}}). Then $D_0=\{\vec{0}\}$, $D_{m-1}\subseteq D_m$ and, by Definition \textcolor{blue}{\ref{def: Rn}}, $R_m^{\vec{c}} = D_m\setminus D_{m-1}$.

    Define $\phi_m: D_m\rightarrow \Z$ by $\phi_m(\vec{v}) = \sum_{i=1}^{m} a_i X_{m+1-i}$, where $(a_i)_{i=1}^\infty$ is the $\vec{c}$-SR of $\vec{v}$. By Definition \textcolor{blue}{\ref{defintion:SR}}, once its leading zeros are deleted, the string $a_1a_2\cdots a_m$ is exactly the string of coefficients of a legal decomposition, where the zero string corresponds to the empty decomposition of $0$. Hence, by \textcolor{blue}{\cite{MW12}}, Theorem 1.1, applied to each interval $[X_j, X_{j+1})$ with $1\leq j\leq m$, the map $\phi_m$ is a bijection from $D_m$ onto $[0, X_{m+1})$, and the legal decomposition of $\phi_m(\vec{v})$ has exactly $G(a)$ summands. Since $0\leq \phi_m(\vec{v})<X_{m+1}$, Lemma \textcolor{blue}{\ref{lem:S_n}} shows that $\phi_m(\vec{v})=S_{m+1}(\vec{v})$ whenever $S_{m+1}$ is defined.

    Fix $j\geq 0$ and let $N_m(j)$ be the number of integers in $[0, X_{m+1})$ whose legal decomposition has exactly $j$ summands. Then $D_m$ contains exactly $N_m(j)$ vectors whose $\vec{c}$-SR has $j$ summands. As $R_n^{\vec{c}} = D_n\setminus D_{n-1}$ and $D_{n-1}\subseteq D_n$,
    \begin{align*}
        |\{\vec{v}\in R_n^{\vec{c}} \ :\ K_n^{\vec{c}}(\vec{v})=j\}| \ =\ N_n(j)-N_{n-1}(j) \ =\ |\{\omega\in [X_n, X_{n+1}) \ :\ K_n(\omega)=j\}|.
    \end{align*}
    Summing over $j$ gives $|R_n^{\vec{c}}| = X_{n+1}-X_n$.
\end{proof}
In the 1950s, Lekkerkerker \textcolor{blue}{\cite{Lekkerkerker51}} answered the question: \textit{On average, how many summands are needed in the Zeckendorf decomposition?}\footnote{Sometimes it takes a while for papers to be widely seen, for example Lekkerkerker's theorem on the number of summands in Zeckendorf decompositions was published 20 years before Zeckendorf's paper! } He later proved that for every $m\in[F_n, F_{n+1})$, as $ n \rightarrow \infty $ the average number of summands needed is $ n/(\phi^2 + 1) $ , where $ \phi = \frac{1 + \sqrt{5}}{2}$ is the golden ratio.
Since Zeckendorf's theorem has been generalized one can naturally ask if Lekkerkerker's theorem still holds for these various generalizations. In \textcolor{blue}{\cite{MW12}} it was proved that the Gaussian convergence is a property that holds in general for a PLRS. More specifically, $\mathbb{E}[K_n]$ and Var$(K_n)$ are of order $n$, and as $n \rightarrow \infty$,
$$\frac{K_n-\mathbb{E}[K_n]}{\sqrt{{\rm Var}(K_n)}}\ \xrightarrow{\ d\ }\ N(0,1)$$
(see \textcolor{blue}{\cite{KKMW11}}, \textcolor{blue}{\cite{BILMT13}} and \textcolor{blue}{\cite{Lekkerkerker51}}). This result is generalized as follows.

\begin{theorem}
    Let $\{\X_n\}_{n \in \Z}$ be a %Positive Linear Vector Recurrence
    $\vec{c}$-recurrence for a weakly decreasing vector $\vec{c}$ where $c_k = 1$. Then $\mathbb{E}[K_n^{\vec{c}}]$ and {\rm Var}$(K_n^{\vec{c}})$ are of order $n$, and as $n \rightarrow \infty$,
    $$\frac{K_n^{\vec{c}}-\mathbb{E}[K_n^{\vec{c}}]}{\sqrt{{\rm Var}(K_n^{\vec{c}})}}\ \xrightarrow{\ d\ }\ N(0,1).$$
\end{theorem}
\begin{proof}
    By Lemma \textcolor{blue}{\ref{lem:Sn&Kn}}, $K_n^{\vec{c}}$ has the same distribution as $K_n$ for the PLRS $(X_n)_{n=1}^\infty$ of Definition \textcolor{blue}{\ref{def:c-rec seq}}; in particular, they have the same mean and variance. The result now follows from \textcolor{blue}{\cite{MW12}}, Theorems 1.2 and 1.3.
\end{proof}

\subsection{Summand Minimality}\label{sec: SM}
We call a representation of $\vec{v}$ \textbf{\textit{summand minimal}} if no other representation of $\vec{v}$ with nonnegative coefficients uses fewer summands. We say that a Positive Linear Recurrence Vector $\X$ is \textbf{\textit{summand minimal} }if for each vector $\vec{v}\in \mathbb{Z}^{k-1}$, there is a $\vtr{c}$-SR of $\vtr{v}$ that is summand minimal. Note that if a $\vtr{c}$-SR does not exist for a vector $\vtr{v}\in \mathbb{Z}^{k-1}$ then by this definition $\vtr{X}$ is not summand minimal. In \textcolor{blue}{\cite{CHHMPV18b}} it is proved that a PLRS is summand minimal if and only if its recurrence vector is weakly decreasing. We show that by use of our main result, we quickly obtain one direction of this result in our multidimensional setting.

\begin{theorem}\label{thm: c minimal}
    If $c_1\geq c_2\geq \cdots \geq c_k=1$ then for any $\vtr{v}\in \mathbb{Z}^{k-1}$, its unique $\vtr{c}$-SR is summand minimal.
\end{theorem}

\begin{proof}

Suppose that $\vec{c}$ is weakly decreasing, so that
\[
c_1\ \geq\ c_2\ \geq\ \cdots\ \geq\ c_k=1.
\]
Let $\vec{v}\in\mathbb{Z}^{k-1}$ and let
\[
\vec{v}\ =\ \sum_{i=1}^{n}b_i\vec{X}_{-i}
\]
be a representation of $\vtr{v}$. We find a $\vtr{c}$-SR of $\vtr{v}$ where the sum of the coefficients in this representation are no larger than $\sum_{i=1}^n b_i$. Note that if $b$ is a $\vtr{c}$-SR then we are done and so we suppose not. Starting with terms of $b$ with index $i=1$ keep adding copies of $\vtr{X}_{-1}$ to $\vtr{0}$ until either you have $b_1\vtr{X}_{-1}$ or until the first instance the resulting representation, say $d_1\vtr{X}_{-1}$, is a $\vtr{c}$-NSR. Now using Proposition \textcolor{blue}{\ref{prop: general cSR}} we may transform $d_1\vtr{X}_{-1}$ into a $\vtr{c}$-SR, say $d^1$, by a finite sequence of carrying and borrowing and with a sum of coefficients no larger than $b_1$. Now continually add copies of $\vtr{X}_{-1}$ until either you have added $b_1-d_1$ many copies or until the first instance the resulting representation, say $d_2\vtr{X}_{-1}$, is a $\vtr{c}$-NSR. By continually applying Proposition \textcolor{blue}{\ref{prop: general cSR}} in this case, we see that we can produce a $\vtr{c}$-SR of $b_1\vtr{X}_{-1}$, say $b^1$, with no more summands.\\

Now focusing on the terms with index $i=2$, keep adding copies of $\vtr{X}_{-2}$ to $b^1$ until either you have added $b_2$ many copies or until the first instance the resulting representation, say $f_1$, becomes a $\vtr{c}$-NSR. As before we may continue applying the transformations of Proposition \textcolor{blue}{\ref{prop: general cSR}} to eventually obtain a $\vtr{c}$-SR of $b_1\vtr{X}_{-1}+b_2\vtr{X}_{-2}$, say $b^2$, with no more summands.\\

The process can be continued until all indices of $b$ with nonzero terms are covered. Hence, the end result of this process is a $\vtr{c}$-SR of $\vtr{v}$ whose sum of coefficients is at most $\sum_{i=1}^n b_i$. This completes the proof.

\end{proof}

\begin{remark}\label{rem: summand minimality}
    While it seems that the converse of Theorem \ref{thm: c minimal} follows by similar techniques as in \textcolor{blue}{\cite{CHHMPV18b}}, we note that the methodology for the scalar setting involves continual carrying to initial terms that are zero. Unfortunately, this cannot be done in the multidimensional setting as defined. Furthermore, while every integer has a unique legal sum over any PLRS, it is unclear whether in the multidimensional setting such a $\vtr{c}$-SR exists for each vector $v\in \mathbb{Z}^{k-1}$. In fact, by using standard carrying and borrowing operations, the algorithm utilized in Proposition \textcolor{blue}{\ref{prop: general cSR}} may fail to terminate and generate vectors with larger and larger number of summands (see Example \textcolor{blue}{\ref{Ex: Bad Algorthm}} and Question \textcolor{blue}{\ref{Qu:c conditions}}).
\end{remark}

\section{Illustrations and Further Research}

Throughout this section $\vtr{c}\in \Z^3$ denotes an arbitrary vector such that $c_3=1$ and $(\vtr{X}_n)_{n\in \Z}$ the corresponding $\vtr{c}$-recursive vector sequence as in Definition \textcolor{blue}{\ref{Def: Rec Vectors}}.

% Credit: Tianyu

\begin{definition}\label{Def: D_n's}

Define

$$D_n\ =\ \left\{ \vtr{v} \in \mathbb{Z}^{2} \mid \vtr{v} = \sum_{i=1}^n a_i \vtr{X}_{-i}\ \ \ \text{where}\ a\textsf{ is a }\vtr{c}\textsf{-SR}  \right\}.$$

\end{definition}

\begin{remark}
     By this definition, $D_n=\{\vtr{0}\}\cup\bigcup_{i=1}^n R_i^{\vec{c}}$ and the number of points in \( D_n \) is \( |D_n| = X_{n+1} \).
\end{remark}

\begin{remark}\label{remark:RnD}
    In general\footnote{Definition \textcolor{blue}{\ref{Def: D_n's}} is naturally extended to higher dimensions.}, for all $n \in \N$,  we have that $R_n^{\vec{c}} = D_n \setminus D_{n-1}.$
\end{remark}

Figures \textcolor{blue}{\ref{fig:D1toD7}} and \textcolor{blue}{\ref{fig: D_10 in (2,1,1) large}} demonstrate how, in the case where $\vtr{c}=(2,1,1)$, each region, $D_n$, is recursively constructed from $D_{n-1}$ and spirals away from $\vtr{0}$.

\begin{figure}[ht]
    \centering
    \includegraphics[width=0.25\textwidth]{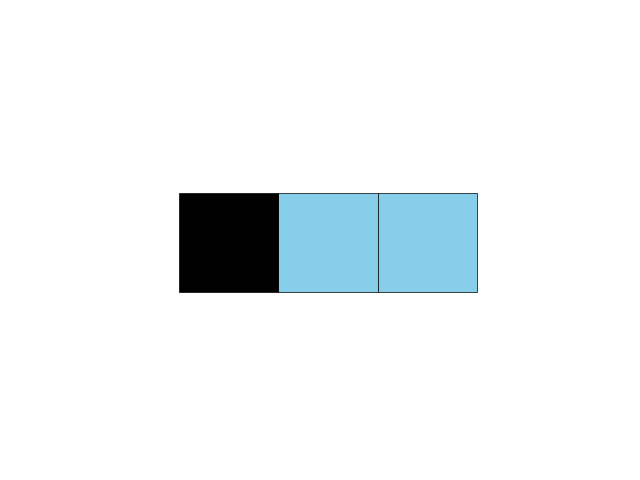}
    \includegraphics[width=0.25\textwidth]{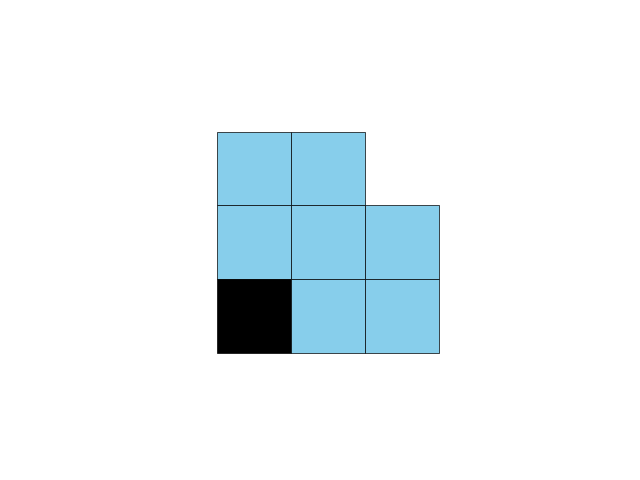}
    \includegraphics[width=0.25\textwidth]{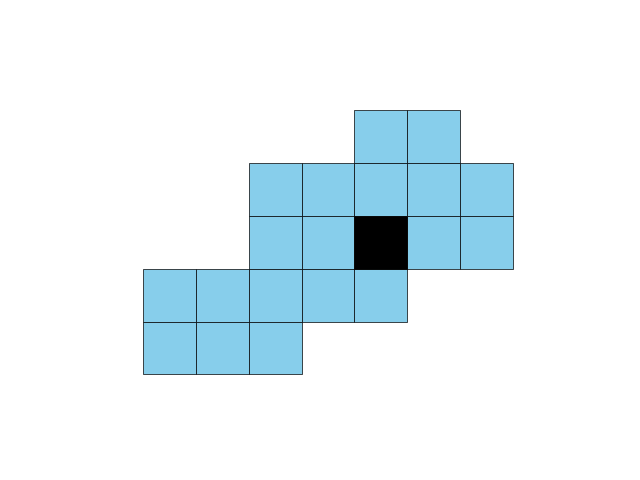}
    \includegraphics[width=0.25\textwidth]{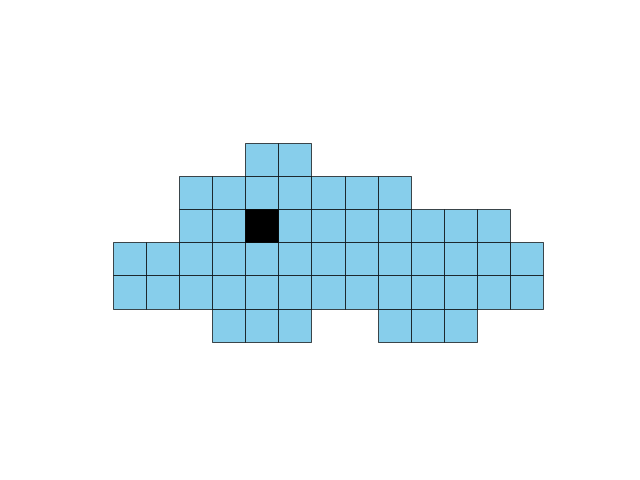}
    \includegraphics[width=0.25\textwidth]{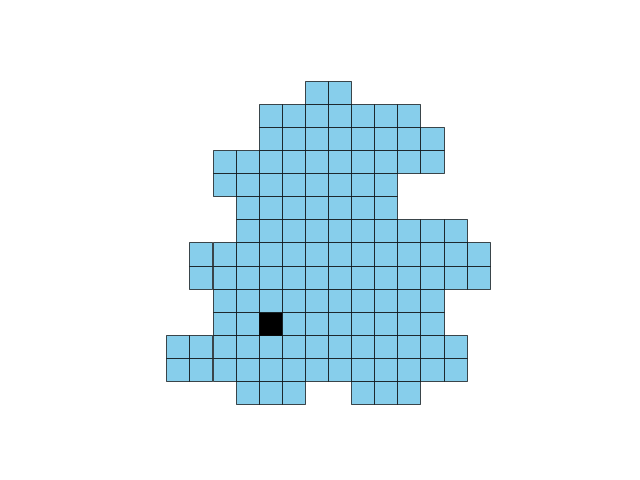}
    \includegraphics[width=0.25\textwidth]{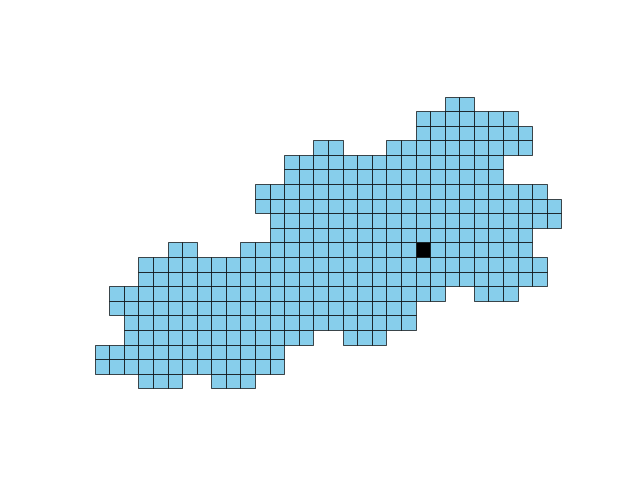}
    \caption{Regions $D_1$,\ \ldots,\ $D_6$ for $\vtr{c}=(2,1,1)$. The black square indicates $\vec{\mathbf{0}} \in \mathbb{Z}^2$.}
    \label{fig:D1toD7}
\end{figure}

\begin{figure}[htp]
    \hspace{-1cm}
    \begin{subfigure}[htp]{.47\textwidth}
        \centering
        \includegraphics[width=1.2\textwidth]{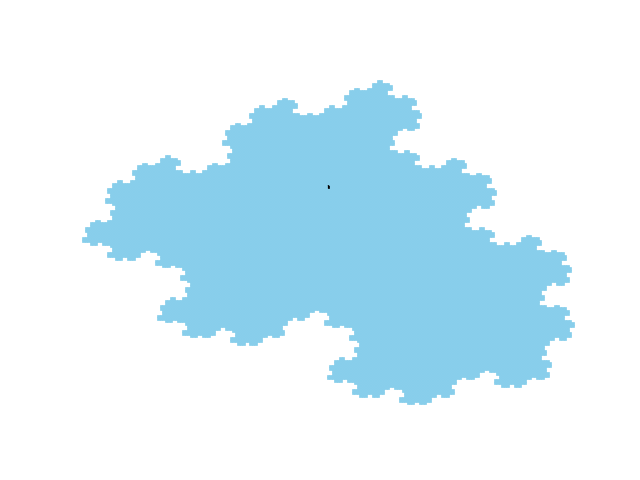}
        %\caption{Region $D_{10}$ in a (2,1,1).}
        \label{fig:D10_in_211}
    \end{subfigure}
 \hfill
 \hspace*{-1.5cm}
    \begin{subfigure}[htp]{.47\textwidth}
        \centering
        %\vspace{4cm}
        \includegraphics[width=1.2\textwidth]{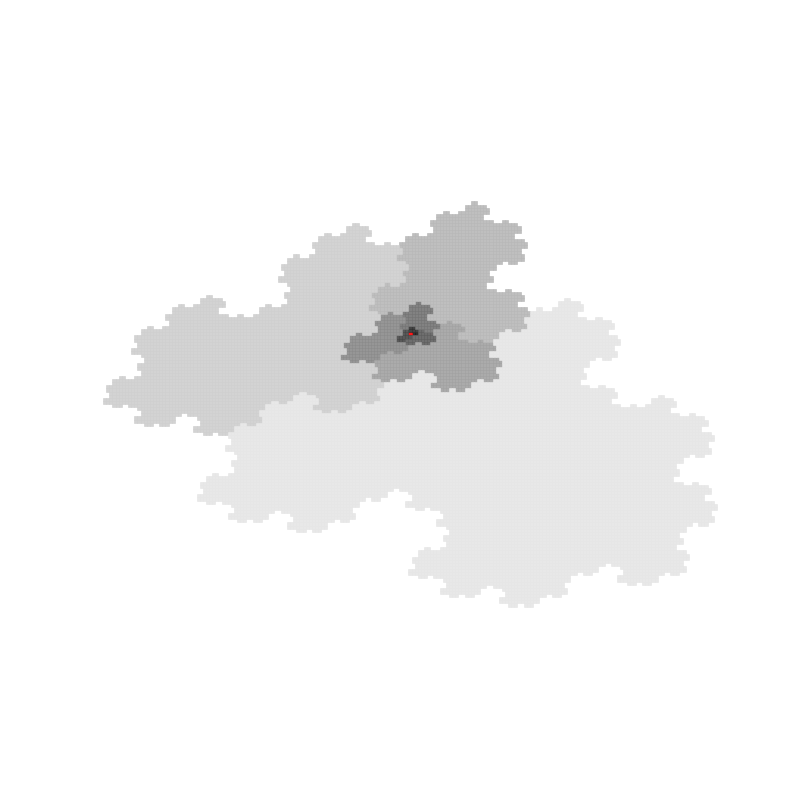}
        %\caption{Regions $R_1$, \ldots ,$R_{10}$ in a (2,1,1).}
        \label{fig:R1_to_R10_in_211}
    \end{subfigure}
    \vspace{-2cm}
    \caption{Region $D_{10}$ and regions $R_1$,\ \ldots,\ $R_{10}$ for $\vtr{c}=(2,1,1)$ respectively.}
    \label{fig: D_10 in (2,1,1) large}
\end{figure}

 However, in many cases where $\vtr{c}$ is not weakly decreasing the sets, $D_n$, display similar behavior. For example, if $\vtr{c}=(1,2,1)$ then Figures \textcolor{blue}{\ref{fig:D1toD7 2}} and \textcolor{blue}{\ref{fig: D_10 in (1,2,1) large}} also seem to be slowly expanding outward in a spiral that appears to eventually encompass all of $\mathbb{Z}^2$. This leads us to our first question.

\begin{figure}[ht]
    \centering
    \includegraphics[width=0.25\textwidth]{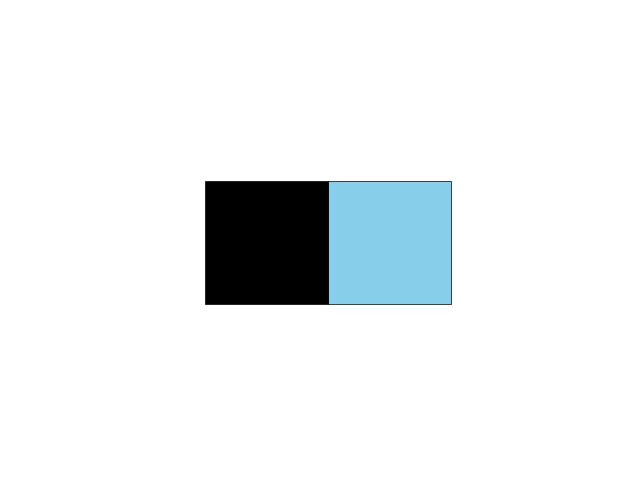}
    \includegraphics[width=0.25\textwidth]{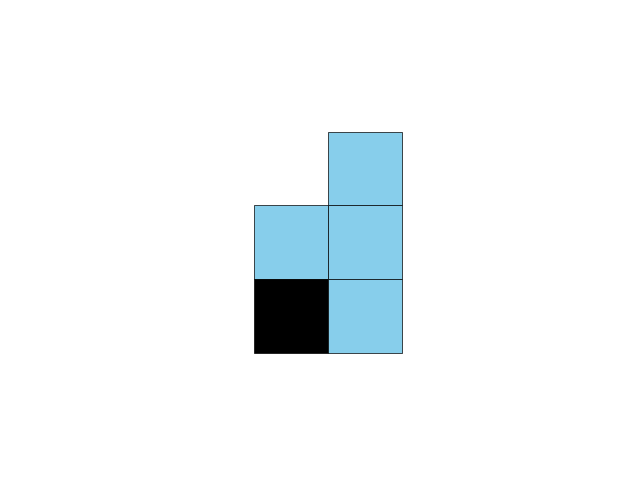}
    \includegraphics[width=0.25\textwidth]{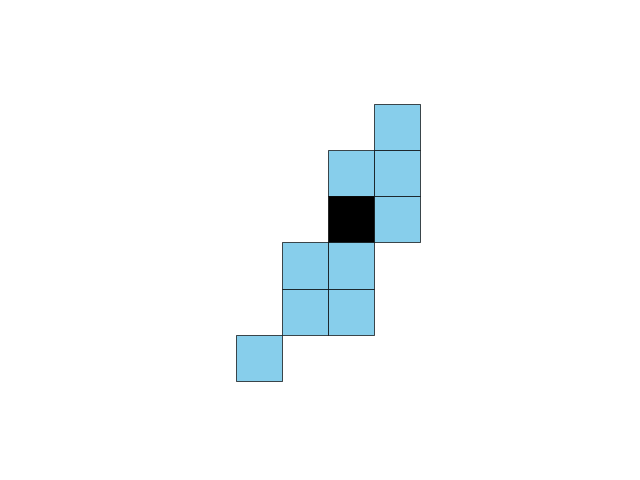}
    \includegraphics[width=0.25\textwidth]{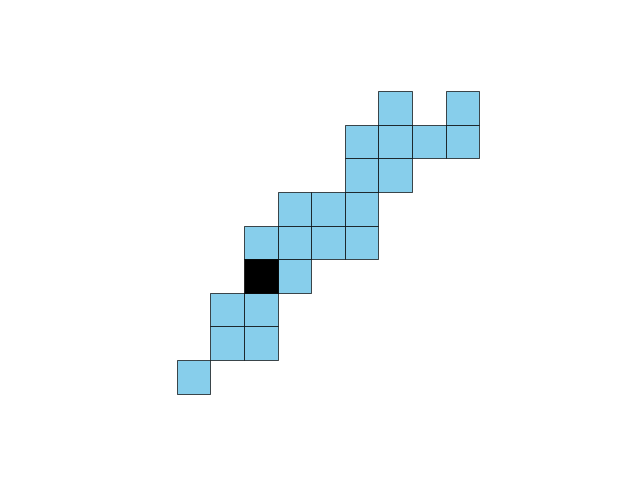}
    \includegraphics[width=0.25\textwidth]{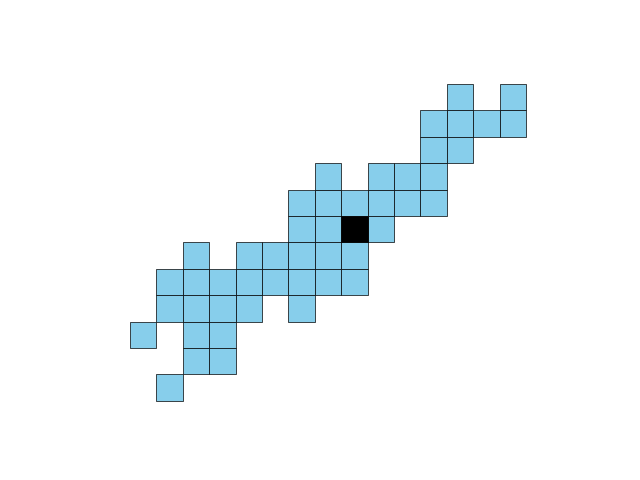}
    \includegraphics[width=0.25\textwidth]{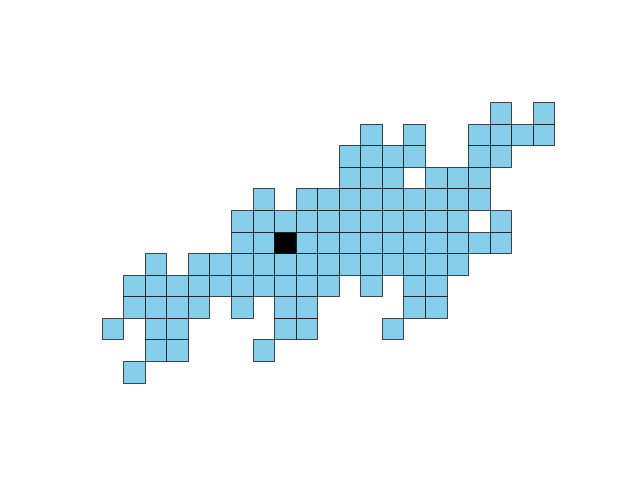}
    \caption{Regions $D_1$,\ \ldots,\ $D_6$ for $\vtr{c}=(1,2,1)$. The black square indicates $\vec{\mathbf{0}} \in \mathbb{Z}^2$.}
    \label{fig:D1toD7 2}
\end{figure}

\begin{figure}[htp]
    \centering
    \begin{subfigure}[htp]{.48\textwidth}
        \centering
        \includegraphics[width=1.2\textwidth]{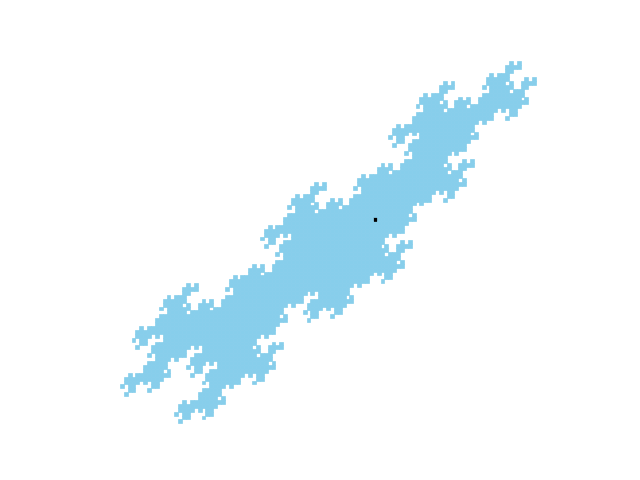}
        \label{fig:D10_in_121}
    \end{subfigure}
    \hfill
    \begin{subfigure}[htp]{.48\textwidth}
        \centering
        \hspace{-.2cm}
        \includegraphics[width=1.2\textwidth]{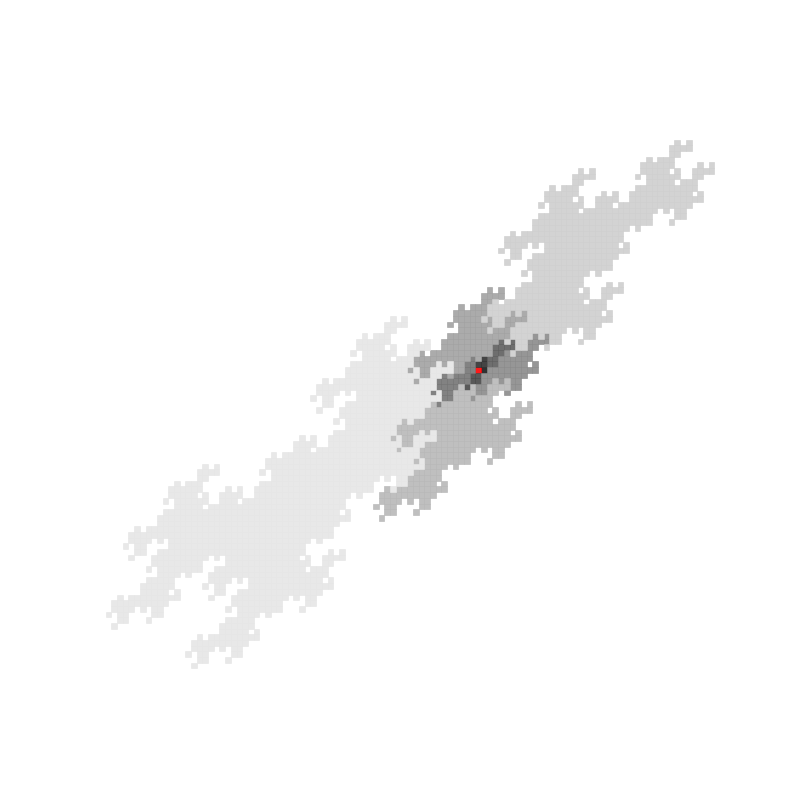}
        \label{fig:R1_to_R10_in_121}
    \end{subfigure}
    \vspace{-2cm}
    \caption{Region $D_{10}$ and regions $R_1$,\ \ldots,\ $R_{10}$ for $\vtr{c}=(1,2,1)$ respectively.}
    \label{fig: D_10 in (1,2,1) large}
\end{figure}

\begin{question}\label{Qu:c conditions}
    Which conditions on the vector $\vtr{c}$ characterize when the algorithm in the proof of Theorem \textcolor{blue}{\ref{thm: main}} terminates?
\end{question}

In the case where at least some of the terms of a vector $\vtr{c}$ are increasing, the algorithm utilized in the proof of Proposition \textcolor{blue}{\ref{prop: general cSR}} can fail to terminate.

\begin{example}\label{Ex: Bad Algorthm}
    Let $\vtr{c}=(1,3,1)$. The string $2$ is a $\vtr{c}$-NSR as the string $1$ is a $\vtr{c}$-SR. We follow the first few steps in the proof of Proposition \textcolor{blue}{\ref{prop: general cSR}}.

    \begin{enumerate}
        \item We borrow once from $2$ to obtain the string $1131$. Note that when reading this string from left to right, the first issue that disallows $1131$ from being a $\vtr{c}$-SR is the digit $3$. Note how in the case where some of the terms of $\vtr{c}$ increase, it is not necessarily so that a carry can always follow a borrow.
        \item As $11$ forms a chunk ($1\leq 1$ and $1<3$) the first digit that presents a problem is $3$ which at most can be $1$. We borrow twice from $3$ into the lower digit. This converts $1131$ to $111362$. Then we can carry into the second term and turn $111362$ into $120052$.
        \item As in the previous step $1200$ forms a chunk and $5$ is the first digit that presents a problem since $5>1$. We borrow 4 times from $5$ and obtain $120016(12)4$ and carry once into the fourth term to obtain $120103(11)4$ and then carry 3 times into the fifth term to obtain $12013021$.
        \item From left to right we have the chunks $120$ and $130$ and then the first digit to present a problem is the seventh term as $2>1$. We borrow 1 time from the seventh term to obtain $1201301231$.
        \item From left to right we have the chunks $120$, $130$, and $12$ and then the first digit to present a problem is the ninth term as $3>1$. We borrow twice from the ninth term to obtain $120130121362$ and carry once into the eighth term to obtain $120130130052$.
    \end{enumerate}

Examining the conversion of the digits $052$ from step (2) to (5) we see that the string $052$ is replaced by $130130052$. Hence, this process never terminates.

\end{example}

Nevertheless, it seems that, if $\vtr{c} = (c_1, c_2 ,c_3)$, $c_3 = 1$ and $c_{j+1} \leq c_j + 1$ for all $1\leq j < 3$, then the algorithm terminates. However, for vectors of length $k\geq 4$ this no longer appears to be the case. From this we make the following conjecture.

\begin{conjecture}\label{con: alg term}
    If $k=3$ the algorithm in the proof of Theorem \textcolor{blue}{\ref{thm: main}} terminates if $\vtr{c}=(c_1,c_2,c_3)$ is such that $c_i\geq 0$ for all $1\leq i< 3$, $c_{i+1}-c_{i}\leq 1$ for all $1\leq i\leq k-1$, and $c_k=1$. Furthermore, if there exists $1\leq i<3$ such that $c_{i+1}-c_i\geq 2$, then there is a $\vtr{c}$-representation for which the algorithm fails to terminate. If $k\geq 4$ the algorithm in the proof of Theorem \textcolor{blue}{\ref{thm: main}} terminates iff $\vtr{c}$ is weakly decreasing.
\end{conjecture}

For those $\vtr{c}$ where a $\vtr{c}$-SR exists for each $\vtr{v}\in \Z^{k-1}$, it would be interesting to find a rate of expansion outward from $\vtr{0}$ that is somehow dependent upon $\vtr{c}$. One such approach would involve attacking the following question.

\begin{question}\label{qu: expanding balls}
    For each $\vtr{v}\in \Z^{k-1}$ and $r\in \mathbb{N}$, define $B^{\infty}_r$ by

    $$B^{\infty}_r(\vtr{v})\ :=\ \{\vtr{w}\in \Z^{k-1}\ :\ \max_{1\leq i\leq k-1}\{|w_i-v_i|\}\leq r\ \}.$$

    \noindent Given a $\vtr{c}$ such that the algorithm in the proof of Theorem \textcolor{blue}{\ref{thm: main}} terminates and an $r\in \N$, what is the minimum $n\in \N$ such that $B^\infty_r(\vtr{0})\subset D_n$?
\end{question}

In addition, as discussed in Remark \textcolor{blue}{\ref{rem: summand minimality}}, it would be interesting to see if summand minimality implies the underlying recurrence vector is weakly decreasing in the multidimensional setting.

\begin{question}\label{qu: summand minimality}
    If $\vtr{c}$ is a recurrence vector that defines a PLRS in the scalar setting and is NOT weakly decreasing, then does the corresponding vector recurrence sequence $\vtr{X}$ fail to be summand minimal?
\end{question}

Lastly, there are many other generalizations of Zeckendorf's Theorem that could be explored in the multidimensional case. One such generalization is \textbf{\textit{$f$-decompositions}} introduced by \textcolor{blue}{\cite{DDKMMV14}}.

\begin{definition}\label{def: f-decomps}
    Given a function $f:\N_0 \rightarrow \N_0$ and a sequence of integers $(a_n)_{n=1}^\infty$, a sum $m = \sum_{i=0}^s a_{n_i}$ of terms of $(a_n)_{n=1}^\infty$ where $s\in \mathbb{N}_0$ and such that $n_0>n_1>\cdots >n_s>1$, is an \textit{\textbf{$f$-decomposition of $m$ using $(a_n)_{n=1}^\infty$}} if for every $a_{n_i}$ in the $f$-decomposition, the previous $f(n_i)$ terms $(a_{n_i - f(n_i)},\ a_{n_i - f(n_i)+1},\  \ldots,\ a_{n_i-1})$ are not in the $f$-decomposition (where negative indices are considered vacuous).
\end{definition}

\begin{question}{\label{Qu: f-decomps}}
    For some family of functions, can one apply a similar strategy as in Theorem \textcolor{blue}{\ref{thm: main}} to generate multidimensional $f$-decompositions?
\end{question}

\section*{Acknowledgments}

We thank our colleagues from the 2025 Polymath Jr program and the anonymous referee whose wonderful suggestions drastically improved the quality and readability of this work. This work was supported by NSF Grant DMS2341670. The results were independently checked using the Aristotle formal verification tool.

\medskip
%AMS classifications available at https://mathscinet.ams.org/mathscinet/freetools/msc-search
\noindent MSC2020: 11A67, 11B39, 11B34.
\end{document}